\documentclass[leqno,11pt, a4]{amsart}
\usepackage[utf8]{inputenc}
\usepackage{amsthm}
\usepackage{amsmath}
\usepackage{mathtools}
\usepackage{enumitem}
\usepackage{amsfonts}
\usepackage{amssymb}
\usepackage{graphicx}
\usepackage{pdfsync}
\usepackage{hyperref}
\usepackage{color}

\def\R{\text{$\mathbb{R}$}}
\def\C{\text{$\mathbb{C}$}}
\newcommand{\di}{\mathrm{d}}

\newcommand{\g}[1]{\mathfrak{#1}}
\newcommand{\Ad}{\mathrm{Ad}}
\newcommand{\ad}{\mathrm{ad}}
\newcommand{\dcp}[2]{\left\langle #1|#2\right\rangle}
\newcommand{\ssp}[2]{\left\langle #1,#2\right\rangle}
\newcommand{\nsp}[2]{\left\langle\!\left\langle #1,#2\right\rangle\!\right\rangle}
\newcommand{\Hess}{\text{H}}
\newcommand {\der}[2]{\frac{d#1}{d#2}}

\newcommand {\re}[1]{\mathfrak{Re}(#1)}

\newtheorem{thm}{Theorem}[section]
\newtheorem{prop}[thm]{Proposition}
\newtheorem{lemma}[thm]{Lemma}
\newtheorem{cor}[thm]{Corollary}

\newtheorem{mainthm}{Theorem}

\theoremstyle{definition}
\newtheorem{rem}[thm]{Remark}
\newtheorem{defini}[thm]{Definition}
\newtheorem{exe}[thm]{Example}

\title[]{A Calabi--Matsushima decomposition for real reductive group actions}
\author{L. Biliotti, A. Minuzzo \& F. Podestà}
\address{Dipartimento di Scienze Matematiche, Fisiche e Informatiche -
          Universit\`a degli Studi di Parma (Italy)}
\email{leonardo.biliotti@unipr.it}
\email{alessandro.minuzzo@unipr.it}
\address{Dipartimento di Matematica e Informatica "Ulisse Dini" - Università degli Studi di Firenze (Italy)}
\email{fabio.podesta@unifi.it}

\begin{document}

\begin{abstract}
Let $Z$ be a compact Kähler manifold endowed with a Hamiltonian action of a compact connected Lie group $U$, and let $G\subset U^{\mathbb C}$ be a real compatible subgroup. Using the gradient map $\mu_\g{p}$ associated with the momentum map, we establish a Calabi–Matsushima type decomposition for the Lie algebra of the stabilizer $G_z$ at a critical point of $f\circ\mu_{\mathfrak p}$, where $f$ is the restriction to $\mathfrak p$ of a suitable $\operatorname{Ad}_U$-invariant strictly convex function on $i\mathfrak u$. More precisely, we show that the isotropy algebra decomposes into eigenspaces corresponding to nonnegative eigenvalues of an adjoint endomorphism, with zero eigenspace given by its reductive part. This extends the classical decomposition for the complexified action of $U^{\mathbb C}$ to the real reductive setting. As an application, we prove that the identity component of the compact stabilizer at critical points of $f\circ\mu_\g{p}$ is a maximal compact subgroup of the identity component of the $G$-stabilizer.
\end{abstract}
\maketitle

\section{Introduction}
Let $(Z,J,\omega)$ be a compact connected Kähler manifold, where $J$ is the complex structure and $\omega$ the symplectic form, and let $U\subset \text{U}(n)$, $n\geq 1$, be a compact connected Lie subgroup of the unitary group, which acts by Kähler isometries on $Z$. This means that the action preserves $J,\ \omega$ and hence also the Riemannian metric $Q(\cdot,\cdot):=\omega(\cdot, J\cdot)$. The \textit{complexification} of $U$ is the \textit{reductive} group $U^\C=\{\exp(i\eta)u\ :\ u\in U,\   \eta\in\g{u}\}\subseteq \text{Gl}(n,\C)$. Since $Z$ is compact, the $U$-action induces a holomorphic action of $U^\C$. We further assume that the $U$-action is \emph{Hamiltonian}, meaning that there exists an $\Ad^*_U$-equivariant map $\mu:Z\to \g{u}^*$ such that 
$\dcp{(\di\mu)_z(v)}{\xi}={\omega_z(v,\xi_Z(z))}$ for all $z\in Z$, $v\in T_zZ$ and $\xi\in\g{u}$, where $\xi_Z(z):=\der{}{t}\exp(t\xi).z|_{t=0}$ is the \emph{infinitesimal generator of the action} and $\dcp{\cdot}{\cdot}:\g{u}^*\times\g{u}\to \R$ is the canonical pairing. Such a map $\mu$ is called a \emph{momentum map} for the $U$-action. 

In a paper by Lee, Sturm and Wang (see Theorem 1.1 of \cite{LSW25}), the authors show that given any smooth, strictly convex, $\Ad_U^*$-invariant function $f:\g{u}^*\to \R$, for all critical points $z\in Z$ of $f\circ \mu$, the subspace $\g{u}^\C_z:=\{\eta\in\g{u}^\C\ : \ \eta_Z(z)=0\}$ admits the following decomposition \begin{equation}\label{Eq1:CMdec}\g{u}^\C_z=\bigoplus_{\lambda\leq 0}\g{u}^\C_{\lambda}\quad \text{where}\quad \g{u}^\C_{\lambda}=\{\eta\in\g{u}^\C_z\ :\ \ad_{i (\di f)_{\mu(z)}}\eta=\lambda \eta\}\ .\end{equation}
In particular, $\g{u}^\C_{0}=\g{u}_z\oplus i\g{u}_z$ is the reductive part of $\g{u}^\C_z$ and $(\di f)_{\mu(z)}$ lies in the center of $\g{u}^\C_{0}$. The decomposition (\ref{Eq1:CMdec}) is referred to as the \emph{Calabi--(Lichnerowicz)--Matsushima decomposition}. As explained in \cite{GRS24}, this result is useful in proving important criteria such as the Székelyhidi Criterion (\cite{GRS24}, Theorem 13.4; see also the original article \cite{S07}). The interest in these results is motivated by their applications in an infinite-dimensional setting, where, for instance, one studies the K\"ahler manifold $\mathcal{J}$ of all integrable almost complex structures $J$ on $Z$ that are compatible with a given K\"ahler form $\omega_0$ (see \cite{Don97, Don017, Fu90}). In this case, the mapping that to each $J$ associates the scalar curvature of the induced Riemannian metric $g_J(\cdot,\cdot):=\omega_0(\cdot, J\cdot)$ plays the role of the momentum map for the action of the group of Hamiltonian diffeomorphisms of $(Z,\omega_0)$ on $\mathcal{J}$.  See \cite{LSW25} for an application of (\ref{Eq1:CMdec}) to these matters. Furthermore, the classical problem of determining \emph{extremal metrics} between the $g_J$ for the so-called \emph{Calabi functional} has its finite dimensional counterpart in the problem of determining critical points of the norm squared of the momentum map on $U^\C$-orbits in $Z$ (see \cite{S14} or \cite{Ca85}). The existence of such metrics is related to the notion of $K$-\emph{stability} (see again \cite{S14}) which, in turn, has its finite dimensional counterpart in the notion of \emph{stability} in \emph{differentiable Geometric Invariant Theory (dGIT)}, as it is discussed in \cite{GRS24}. 

From this perspective, it is possible to generalise the dGIT setting by studying instead of the action of $U^\C$, the action of a closed real subgroup $G\subseteq U^\C$. This subgroup must be compatible with the Cartan decomposition $U^\C=U\exp({i\g{u}})$, in the sense that there is a decomposition $G=K\exp(\g{p})$, with $K:=U\cap G$ compact subgroup and $\text{Lie}(G)=\g{g}=\g{k}\oplus \g{p}$, with $\text{Lie}(K)=\g{k}$ and $\g{p}:=i\g{u}\cap \g{g}$ being an $\Ad_K$-invariant subspace. The group $G$ is then called a \emph{real compatible} (or \emph{real reductive}) subgroup of $U^\C$. Thanks to the work of Heinzner, Schwarz et al., many of the fundamental results of classical dGIT were extended to this more general framework (see \cite{HS07,HS08,HS010}). Moreover, using the \emph{gradient map} (see Definition \ref{D:GradMap}), different classical notions of stability were generalised to the case of real reductive groups in \cite{BZ017,BW023,BW024}. 

Following this line of investigation, the aim of this note is to establish a
Calabi--Matsushima decomposition for real reductive subgroups
$G \subseteq U^{\mathbb C}$. We observe that this result does not follow formally from the complex case. A first essential difference is that, in general, $i\mathfrak{p}$ is not contained in $\mathfrak{k}$, and, similarly, $i\mathfrak{k}$ is not necessarily contained in $\mathfrak{p}$.\\

\noindent\textbf{Theorem A.} \textit{Let $G\subseteq U^\C$ be a real reductive subgroup with Cartan decomposition $G=K\exp(\g{p})$, $\g{g}=\g{k}\oplus \g{p}$.
 Let $\mu_\g{p}:Z\to \g{p}$ be the gradient map associated with an $\Ad_U^*$-equivariant momentum map $\mu:Z\to \g{u}^*$ (see Definition \ref{D:GradMap}), and let $f:\g{p}\to\R$ be an $\Ad_K$-invariant, strictly convex function. If $z\in Z$ is a critical point of $f\circ \mu_\g{p}:Z\to \R$, and if $f$ admits a strictly convex, $\Ad_U$-invariant extension to $i\g{u}$, which is $(\g{k},\g{p})$-compatible at $\mu_\g{p}(z)$ (see Definition \ref{D:kpComp}), there exists a direct sum decomposition of $\g{g}_z=\{\xi\in\g{g}\ :\ \xi_Z(z)=0\}$
 \begin{equation}\label{Eq:CMdec2}
     \g{g}_z=\bigoplus_{\lambda\geq 0}\g{g}_\lambda\ ,
 \end{equation}
where $\g{g}_\lambda:=\{\xi\in\g{g}_z\ :\ \ad_{(\nabla f)_{\mu_\g{p}(z)}}\xi=\lambda\xi\}$, $\g{g}_0=\g{k}_z\oplus \g{p}_z$, with $\g{k}_z=\g{g}_z\cap\g{k}$ and $\g{p}_z=\g{g}_z\cap \g{p}$.}\\


We observe that (\ref{Eq:CMdec2}) reduces to (\ref{Eq1:CMdec}) for the particular case where $G=U^\C$, the only difference being an opposite sign convention (see Remark \ref{R:ComplexCase} and Corollary \ref{C:CorollLSW25}). This sign difference is discussed in more detail in Remark \ref{R:SignOp}.

As we shall see later, an example of $\Ad_K$-invariant, strictly convex function on $\g{p}$, with a strictly convex, $(\g{k},\g{p})$-compatible extension is provided by the norm squared induced by the scalar product of Definition \ref{D:Scal1} below, restricted to $\g{p}$. This is not the only one: we provide different instances of non quadratic examples (see Example \ref{E:ExeF}). Interestingly, in the complex case $G=U^\C$, the $(\g{k,\g{p}})$-compatibility is trivially satisfied by any function on $i\g{u}$ (Remark \ref{R:ComplexCase}).

In the end, we will provide a non trivial application of (\ref{Eq:CMdec2}). More precisely, it is proved in \cite{GRS24} that the identity component of the stabilizer of a critical point $z\in Z$ of the momentum map $\mu:Z\to\g{u}^*$, which we denote by $U_z^o$, is equal to $(U^\C_z)^o\cap U$ and is a maximal compact subgroup of the identity component of the complex stabilizer $(U^\C_z)^o$. This result is in fact used in the proof of the above-mentioned Székelyhidi Criterion. We will see that an analogous statement holds in the case of a real reductive subgroup $G$ of $U^\C$ (see Theorem \ref{mainthm2}).
However, the proof of our more general result is substantially different from the one available in the complex case $G=U^{\mathbb C}$. Indeed, some of the structural properties of complex stabilizers used in the classical argument are no longer available for a general real reductive subgroup.
Nevertheless, the Calabi--Matsushima decomposition established above and a series of finer arguments allow us to derive the corresponding maximal compactness result in the real reductive setting. We discuss in Remark \ref{R:Salamon}, in further detail, why the proof for $G=U^{\mathbb C}$ does not directly extend to this more general situation.

\section{Preliminaries}

We start by defining a suitable inner product on $\g{u}^\C=\g{u}\otimes_\R\C=\g{u}\oplus i\g{u}$.

\begin{defini}\label{D:Scal1} Let $\ssp{\cdot}{\cdot}:\g{u}^\C\times\g{u}^\C\to \R$ be the $\R$-bilinear map defined by $$\ssp{X}{Y}:=\begin{cases}
-\re{B(X,Y)}\quad\ \text{if}\quad X,Y\in \g{u}\\
\ \ \ \re{B(X,Y)}\quad\text{if}\quad X,Y\in i\g{u}\\
\ \ \  \re{B(X,Y)} \ \ \ \text{if}\quad X\in \g{u}, Y\in i\g{u}\ ,
\end{cases}$$
where $\re{\cdot}$ denotes the real part and $B:\g{u}^\C\times\g{u}^\C\to \C$ is the $\C$-bilinear form given by
$$B(X,Y):=\text{Tr}(XY)\ .$$
\end{defini}

  Notice that $B(X,X)<0$ for all $X\neq 0\in\g{u}$ and $B(Y,Y)>0$ for all $Y\neq 0\in i\g{u}$, and that $B$ is manifestly $\Ad_{U}$-invariant because of the cyclicity of the trace. We say that such a bilinear form $B$ is of \textit{Euclidean type}. It follows that $\ssp{\cdot}{\cdot}$ is an Euclidean $\Ad_U$-invariant scalar product on $\g{u}^\C$. In particular, we have that the splitting $\g{u}^\C=\g{u}\oplus i\g{u}$ is orthogonal with respect to $\ssp{\cdot}{\cdot}$.

Let $G\subseteq U^\C$ be a real compatible subgroup with Cartan decomposition $G=K\exp(\g{p})$, with $K:=G\cap U$. At the Lie algebra level we have $\text{Lie}(G)=\g{g}=\g{k}\oplus \g{p}$, where $\g{k}=\text{Lie}(K)$ and $\g{p}:=\g{g}\cap i\g{u}$ is $\Ad_K$-invariant. Let $[\cdot,\cdot]:\g{u}\times\g{u}\to \g{u}$ denote the Lie bracket of $\g{u}$, which we extend $\C$-linearly to $\g{u}^\C$. If $\Ad:G\to \text{Aut}(\g{g})$, $g\mapsto \Ad_g$ is the adjoint representation, we have that $(\di (\Ad))_e:\g{g}\to \text{Der}(\g{g})$ and $(\di (\Ad))_e(\xi)(\eta)=[\xi,\eta]$. We will also use the notation $\ad_\xi\eta:=[\xi,\eta]$. It is well known that the following Cartan relations hold: \begin{equation}\label{Eq:CartanRel}[\g{k},\g{k}]\subseteq \g{k},\quad [\g{k},\g{p}]\subseteq \g{p},\quad [\g{p},\g{p}]\subseteq \g{k}\ .\end{equation}

Let $\#:(\g{u}^\C)^*\to \g{u}^\C$ be the duality induced by $\ssp{\cdot}{\cdot}$, and for $\alpha\in \g{u}^*$, we define $i\alpha\in(i\g{u})^*$ by $\dcp{i\alpha}{\xi}=\dcp{\alpha}{i\xi}$ for all $\xi\in i\g{u}$. In particular, $i\circ\#=-\#\circ i:\g{u}^*\to i\g{u}$. 

\begin{defini}\label{D:GradMap}
 Let $\mu:Z\to\g{u}^*$ be an $\Ad_U^*$-equivariant momentum map relative to the isometric holomorphic action of $U$ on $(Z,J,\omega)$. We define the \textit{gradient map} associated to $\mu$ as the map $\mu_\g{p}:Z\to \g{p}$ given by $$\mu_\g{p}(z):=\pi(\#i\mu(z))\ ,$$ 
 where $\pi:\g{u}^\C\to \g{p}$ is the orthogonal projection onto $\g{p}$ with respect to $\ssp{\cdot}{\cdot}$. 
\end{defini}

\begin{prop}\label{P:Prop1}The following elementary facts hold true.
  \begin{itemize}
      \item [(i)] The map $\mu_\g{p}$ is $\Ad_K$-equivariant, i.e. for all $k\in K$ and $z\in Z$ we have
$$\mu_\g{p}(k.z)=\Ad_k(\mu_\g{p}(z))\ ,$$
where by $k.z$ we denote the action of $k\in K$ on $z\in Z$;\\
\item[(ii)]If $\xi\in \g{p}$ and $\xi_Z(z):=\der{}{t}\exp(t\xi).z|_{t=0}$ is the infinitesimal generator, we have
$$\ssp{(\di \mu_\g{p})_z(\cdot)}{\xi}=Q_z(\cdot,\xi_Z(z))\ ,$$
that is equivalent to saying
$$(\nabla\mu_\g{p}^\xi)(z)=\xi_Z(z)\ ,$$
where we defined $\mu_\g{p}^\xi:=\ssp{\mu_\g{p}(z)}{\xi}$, and where $\nabla\mu_\g{p}^\xi$ is the Riemannian gradient of $\mu_\g{p}^\xi$ 
relative to the Riemannian metric $Q(\cdot,\cdot):=\omega(\cdot, J\cdot)$. 
  \end{itemize}  

\end{prop} 

\begin{proof}
$(i)$ By definition $\mu_\g{p}:=\pi\circ\#\circ i\circ \mu$. We have that $\mu\circ k.=\Ad_k^*\circ \mu$ by $\Ad^*_K$-equivariance of $\mu$. The multiplication by $i$ commutes with $\Ad^*_k$ and therefore we are left to show that $\pi\circ\#\circ\Ad^*_k=\Ad_k\circ\pi\circ\#$.

For $\#\circ\Ad_k^*=\Ad_k\circ\#$, we recall that, by definition, $\dcp{\Ad_k^*\cdot}{\cdot}=\dcp{\cdot}{\Ad_{k^{-1}}\cdot}$, and hence, for all $\xi\in\g{u}$, $$\ssp{\#(\Ad^*_{k}\mu)}{\xi}=\dcp{\Ad_k^*\mu}{\xi}=\dcp{\mu}{\Ad_{k^{-1}}\xi}=\ssp{\#\mu}{\Ad_{k^{-1}}\xi}=\ssp{\Ad_k(\#\mu)}{\xi}\ .$$
Finally, $\Ad_k$ commutes with $\pi$, because $\pi$ depends only on $\ssp{\cdot}{\cdot}$, which is $\Ad_U$-invariant, and $\g{p}$ is $\Ad_K$-invariant.

\noindent $(ii)$ By definition of $\mu_\g{p}$ and since $\xi\in \g{p}$, we have $$\ssp{\mu_{\g{p}}}{\xi}=\ssp{\pi(\#(i\mu))}{\xi}=\ssp{\#(i\mu)}{\xi}=\dcp{i\mu}{\xi}=\dcp{\mu}{i\xi}\ .$$

\noindent By noticing that $i\xi\in\g{u}$, since $\mu$ is a momentum map, we have that for all $z\in Z$ and $v\in T_z Z$ 
$$\ssp{(\di \mu_\g{p})_z(v)}{\xi}=\dcp{(\di \mu)_z(v)}{i\xi}=\omega_z(v,(i\xi)_Z(z))\ .$$
\noindent Now, since the action extends to a holomorphic action of $U^\C$ (because $Z$ is compact), we have that $(i\xi)_Z=J\xi_Z$ and therefore,
$$\ssp{(\di \mu_\g{p})_z(v)}{\xi}=\omega_z(v,J_z(\xi_Z(z)))=Q_z(v,\xi_Z(z))\ .$$

\end{proof}

\begin{prop}\label{P:Diffpz} Let $f:\g{p}\to \R$ be a smooth function and $\mu_\g{p}:Z\to \g{p}$ be the gradient map. Then $z\in Z$ is a critical point of $f\circ \mu_\g{p}$ if and only if $$\#(\di f)_{\mu_\g{p}(z)}\in \g{p}_z\ ,$$
where $\g{p}_z=\g{g}_z\cap\g{p}=\{\xi\in \g{p}\ :\ \xi_Z(z)=0\}$. 
\end{prop}
\begin{proof}
If $z\in Z$ is critical for $f\circ\mu_\g{p}$, we have that for all $v\in T_z Z$ $$0=(\di(f\circ\mu_\g{p}))_z(v)=\dcp{(\di f)_{\mu_\g{p}(z)}}{(\di\mu_\g{p})_z(v)}=\ssp{\#(\di f)_{\mu_\g{p}(z)}}{(\di\mu_\g{p})_z(v)}\ .$$
Since $\#(\di f)_{\mu_\g{p}(z)}\in \g{p}$, we can use item $(ii)$ of Proposition \ref{P:Prop1} to conclude that
$$\ssp{\#(\di f)_{\mu_\g{p}(z)}}{(\di\mu_\g{p})_z(v)}=Q_z(v,(\#(\di f)_{\mu_\g{p}(z)})_Z(z))=0\ ,$$
for all $v\in T_zZ$, which tells us precisely that $\#(\di f)_{\mu_\g{p}(z)}\in \g{p}_z$.

\end{proof}

\begin{rem}\label{R:Diffsspz} We will consider the case of $f:\g{p}\to \R$, given by $f(\xi):=\frac{1}{2}\ssp{\xi}{\xi}$. One easily recognize that $(\di f)_{\mu_\g{p}(z)}=\#^{-1}(\mu_\g{p}(z))$ and therefore, that $z\in Z$ is critical for $f\circ\mu_\g{p}$ if and only if $$\#(\di f)_{\mu_\g{p}(z)}=\mu_\g{p}(z)\in \g{p}_z\ .$$  

In the following, we shall adopt a simpler notation by writing, for a general $F:i\g{u}\to \R$, 
$$(\nabla F)_\xi:=\#(\di F)_\xi\ ,$$
which is to be interpreted as a vector field on $i\g{u}$ computed at $\xi\in i\g{u}$. 
\end{rem}

\section{Calabi--Matsushima decomposition}
In this section, we are going to prove the Calabi--Matsushima decomposition for real reductive groups (Theorem \ref{mainthm} below), which will be derived from two technical lemmas. We begin with a definition.

\begin{defini}\label{D:kpComp}
    Let $F:i\g{u}\to\R$ be a smooth function and let $(\Hess F)_\zeta$ denote the Hessian of $F$ at $\zeta\in i\g{u}$. We say that $F$ is $(\g{k},\g{p})$\emph{-compatible} at $\zeta$ if the $\ssp{\cdot}{\cdot}$-symmetric endomorphism $\#(\Hess F)_\zeta:=\#\circ (\Hess F)_\zeta:i\g{u}\to i\g{u}$ preserves $\g{p}$ and $i\g{k}$.
    $$\#(\Hess F)_\zeta(\g{p})\subseteq \g{p},\quad \#(\Hess F)_\zeta(i\g{k})\subseteq i\g{k}\ .$$

    \begin{rem}\label{R:tildeF} 
      If we let $\tilde F:=F\circ i:\g{u}\to \R$, then $i\circ \#(\Hess F)_\zeta\circ i= -\#(\Hess \tilde F)_{-i \zeta}$ is an endomorphism of $\g{u}$. It follows that $\#(\Hess F)_\zeta(i\g{k})\subseteq i\g{k}$ is equivalent to $\#(\Hess \tilde F)_{-i \zeta}(\g{k})\subseteq \g{k}$.
    \end{rem}
\end{defini}

\begin{rem}\label{R:ComplexCase}
   We note that if $G=U^\C$, hence $\g{k}=\g{u}$ and $\g{p}=i\g{u}$, then any smooth function $F:i\g{u}\to \R$ is $(\g{k},\g{p})$-compatible at every point of $i\g{u}$. 
\end{rem}

A trivial example of a strictly convex, $\Ad_U$-invariant, $(\g{k,\g{p}})$-compatible function on $i\g{u}$ is, of course, given by the norm induced by the scalar product of Definition \ref{D:Scal1} restricted to $i\g{u}$. A less trivial, non quadratic function of this kind is provided by the following concrete example.

\begin{exe}\label{E:ExeF} Consider $U=\text{U}(2)$, so that $U^\C=\text{Gl}(2,\C)$. The group $G=\text{Gl}^+(2,\R)\subset \text{Gl}(2,\C)$ is then a real compatible subgroup, with Cartan decomposition $G=\text{SO}(2)\text{Sym}^{+}(2,\R)$. In this case we have $\g{u}=\g{u}(2)=\text{Lie(\text{U}(2))}$, hence $\g{u}^\C=\g{gl}(2,\C)=\g{u}(2)\oplus i\g{u}(2)$ and $\g{g}=\g{gl}(2,\R)=\g{k}\oplus \g{p}=\g{so}(2)\oplus \text{Sym}(2,\R)$. It is straightforward to check that the function $F:i\g{u}\to \R$, defined by
    $$F(X):=\frac{1}{2}\ssp{X}{X}+\sum_{j=2}^k a_j\re{\text{Tr}(X)}^{2j}, \quad a_j\geq 0, k\in \mathbb{N}\ ,$$
    is convex, $\Ad_U$-invariant and $(\g{k},\g{p})$-compatible (at every $X\in i\g{u}$) in the sense of Definition \ref{D:kpComp}. Clearly, if $a_j>0$ for some $j$, then $F$ is a non-quadratic example.
    \end{exe}


    

    
    Now, we have all the ingredients to state the main theorem.

\begin{mainthm}\label{mainthm}
  Let $G\subseteq U^\C$ be a real reductive subgroup with Cartan decomposition $G=K\exp(\g{p})$, $\g{g}=\g{k}\oplus \g{p}$.
 Let $\mu_\g{p}:Z\to \g{p}$ be the gradient map associated with an $\Ad_U^*$-equivariant momentum map $\mu:Z\to \g{u}^*$, and let $f:\g{p}\to\R$ be an $\Ad_K$-invariant, strictly convex function. If $z\in Z$ is a critical point of $f\circ \mu_\g{p}:Z\to \R$, and $f$ admits a strictly convex, $\Ad_U$-invariant extension to $i\g{u}$, which is $(\g{k},\g{p})$-compatible at $\mu_\g{p}(z)$, then there exists a direct sum decomposition of $\g{g}_z=\{\xi\in\g{g}\ :\ \xi_Z(z)=0\}$
$$\g{g}_z=\bigoplus_{\lambda\geq 0}\g{g}_\lambda\ ,$$
where $\g{g}_\lambda:=\{\xi\in\g{g}_z\ :\ \ad_{(\nabla f)_{\mu_\g{p}(z)}}\xi=\lambda\xi\}$, $\g{g}_0=\g{k}_z\oplus \g{p}_z$, with $\g{k}_z=\g{g}_z\cap\g{k}$ and $\g{p}_z=\g{g}_z\cap \g{p}$. 
\end{mainthm}

We immediately notice that, by Remark \ref{R:ComplexCase}, we have the following corollary.

\begin{cor}\label{C:CorollLSW25}[C.f. Theorem 1.1 of \cite{LSW25}]
    Let $G=U^\C$, let $\mu_{p}:Z\to i\g{u}$ be the gradient map associated with an $\Ad_U^*$-equivariant momentum map $\mu:Z\to \g{u}^*$, and let $f:i\g{u}\to\R$ be an $\Ad_U$-invariant, strictly convex function.
    If $z\in Z$ is a critical point of $f\circ \mu_{p}:Z\to \R$, we have a direct sum decomposition of $\g{u}^\C_z=\{\xi\in\g{u}^\C\ :\ \xi_Z(z)=0\}$
$$\g{u}^\C_z=\bigoplus_{\lambda\geq 0}\g{u}^\C_\lambda\ ,$$
where $\g{u}^\C_\lambda:=\{\xi\in\g{u}^\C_z\ :\ \ad_{(\nabla f)_{\mu_\g{p}(z)}}\xi=\lambda\xi\}$, $\g{u}^\C_0=\g{u}_z\oplus i\g{u}_z$, with $\g{u}_z=\g{u}^\C_z\cap\g{u}$ and $i\g{u}_z=\g{u}^\C_z\cap i\g{u}$.
\end{cor}

In order to prove Theorem \ref{mainthm}, we need the following definition. 

\begin{defini}\label{def1} Let $F:i\g{u}\to\R$ be a smooth $\Ad_U$-invariant, strictly convex function. We define a positive definite inner product for every $\zeta\in i\g{u}$, that we denote as $\nsp{\cdot}{\cdot}_\zeta:\g{u}^\C\times \g{u}^\C\to\R$, by

\begin{equation}\label{Eq:ShiftScal}
    \nsp{\xi_1+i\eta_1}{\xi_2+i\eta_2}_\zeta:=\ssp{i\xi_1}{(\#(\Hess F)_{\zeta})^{-1} (i\xi_2)}+\ssp{i\eta_1}{(\#(\Hess F)_{\zeta})^{-1} (i\eta_2)}\ .
\end{equation}
\end{defini}

\begin{rem}
Following Remark \ref{R:tildeF}, we note that, if $\tilde F:=F\circ i$, we can write
$$\nsp{\xi_1+i\eta_1}{\xi_2+i\eta_2}_\zeta=\ssp{\xi_1}{(\#(\Hess \tilde F)_{-i \zeta})^{-1} (\xi_2)}+\ssp{i\eta_1}{(\#(\Hess F)_{\zeta})^{-1} (i\eta_2)}\ .
$$ The fact that for all $\zeta\in i\g{u}$ this is a scalar product comes from the fact that $F$ is strictly convex, i.e., $(\Hess F)_{\zeta}$ is positive definite. This also allows us to invert the isomorphism of $i\g{u}$ given by $\#\circ (\Hess F)_{\zeta}$, still obtaining a scalar product. Moreover, the splitting $\g{u}^\C=\g{u}\oplus i\g{u}$ is still orthogonal with respect to $\nsp{\cdot}{\cdot}_\zeta$ for all $\zeta\in i\g{u}$. Finally, we observe that if $G=U^\C$, hence $\g{k}=\g{u}$ and $\g{p}=i\g{u}$, the scalar product (\ref{Eq:ShiftScal}) is the real part of the Hermitian product induced on $\g{u}^\C$ by the scalar product of $\g{u}$ given by $\ssp{\cdot}{(\#(\Hess \tilde F)_{-i \zeta})^{-1} (\cdot)}$.
\end{rem}

In the next technical lemmas, we will often use the following elementary identities.
$$\dcp{\ad_\xi^*\alpha}{\eta}=-\dcp{\alpha}{\ad_\xi\eta}\qquad \ \ \forall\xi,\eta\in\g{u}^\C,\ \forall\alpha\in{(\g{u}^\C)}^*\ ;$$
$$\ssp{\ad_\xi\zeta}{\eta}=-\ssp{\zeta}{\ad_\xi\eta}\ \ \qquad \forall\xi\in\g{u},\ \forall\eta,\zeta\in\g{u}^\C\ .\ \quad\ $$
\noindent Finally, since by definition
$\ssp{\xi}{\eta}=\ssp{i\xi}{i\eta}$ for all $\xi,\eta\in \g{u}$, we can observe that if $\zeta,\xi\in i\g{u}$ and $\eta\in \g{u}$, then 
$$\ssp{\ad_{\zeta}\xi}{\eta}=-\ssp{\ad_{i\zeta}i\xi}{\eta}=\ssp{i\xi}{\ad_{i\zeta}\eta}=\ssp{\xi}{\ad_\zeta\eta}\ .$$

\begin{lemma}\label{lemma1} For all $\xi\in\g{u}$, $\eta,\zeta\in i\g{u}$, and $F:i\g{u}\to \R$ any $\Ad_U$-invariant function, for every $t\in\R$ we have that 
$$\dcp{(\di F)_{\Ad_{\exp(t\xi)}\eta}}{\Ad_{\exp(t\xi)}\zeta}=\dcp{(\di F)_\eta}{\zeta}\ ,$$
i.e., the quantity on the LHS actually does not depend on $t$.
\end{lemma}
\begin{proof}
The property follows directly from the $\Ad_U$-invariance of $F$.
$$\dcp{(\di F)_{\Ad_{\exp(t\xi)}\eta}}{\Ad_{\exp(t\xi)}\zeta}=\der{}{s}F(\Ad_{\exp(t\xi)}\eta+s\Ad_{\exp(t\xi)}\zeta)\big|_{s=0}=$$
$$=\der{}{s}F(\Ad_{\exp(t\xi)}(\eta+s\zeta))\big|_{s=0}=\der{}{s}F(\eta +s\zeta)\big|_{s=0}=\dcp{(\di F)_\eta}{\zeta}\ .\quad$$
\end{proof}

\begin{lemma}\label{lemma2} Let $f:\g{p}\to \R$ be an $\Ad_K$-invariant, strictly convex function and let $F:i\g{u}\to \R$ be a strictly convex, $\Ad_U$-invariant extension of $f$ to $i\g{u}$, which is $(\g{k},\g{p})$-compatible at $\mu_{\g{p}}(z)$. For all $\xi,\eta\in \g{g}$, we have the following identity

$$ \nsp{\ad_{(\nabla f)_{\mu_\g{p}(z)}}\xi}{\eta}_{\mu_\g{p}(z)}=\ssp{\ad_{\mu_\g{p}(z)}\xi}{\eta}\ .$$
\end{lemma}

\begin{proof}
Clearly, $(\nabla f)_{\mu_\g{p}(z)}\in \g{p}$ and therefore, 
 thanks to the properties of the Cartan decomposition (\ref{Eq:CartanRel}) and the orthogonality of the splitting with respect to the shifted scalar product (\ref{Eq:ShiftScal}), we get
\begin{equation}\label{Eq:Spliteq}
    \nsp{\ad_{(\nabla f)_{\mu_\g{p}(z)}}\xi}{\eta}_{\mu_\g{p}(z)}=\nsp{\ad_{(\nabla f)_{\mu_\g{p}(z)}}\xi^\g{k}}{\eta^\g{p}}_{\mu_\g{p}(z)}+\nsp{\ad_{(\nabla f)_{\mu_\g{p}(z)}}\xi^\g{p}}{\eta^\g{k}}_{\mu_\g{p}(z)}\ .
\end{equation}

Consider the term $\ad_{(\nabla f)_{\mu_\g{p}(z)}}\xi^\g{k}=-\ad_{\xi^\g{k}}(\nabla f)_{\mu_\g{p}(z)}$ and write $\zeta:=\mu_\g{p}(z)$ for simplicity. We then observe that for all 
$\sigma\in{\g{p}}$, thanks to Lemma \ref{lemma1} and since $\ad_{\xi^\g{k}}\sigma\in \g{p}$, we have 

\begin{equation}\label{Eq:ParteFacile}
    \begin{split}
        &\ssp{\ad_{\xi^\g{k}}(\nabla f)_\zeta}{\sigma}=-\ssp{(\nabla f)_\zeta}{\ad_{\xi^\g{k}}\sigma}=-\ssp{(\nabla F)_\zeta}{\ad_{\xi^\g{k}}\sigma}=\\
        &=-\dcp{(\di F)_\zeta}{\ad_{\xi^\g{k}}\sigma}=-\dcp{(\di F)_\zeta}{\der{}{t}\Ad_{\exp(t\xi^{\g{k}})}\sigma\big|_{t=0}}=\\
        &=-\der{}{t}\dcp{(\di F)_{\Ad_{\exp(t\xi^{\g{k}})}\zeta}}{\Ad_{\exp(t\xi^{\g{k}})}\sigma}\big|_{t=0}+\dcp{\der{}{t}(\di F)_{\Ad_{\exp(t\xi^{\g{k}})}\zeta}\big|_{t=0}}{\sigma}=\\
        &=\dcp{\der{}{t}(\di F)_{\Ad_{\exp(t\xi^{\g{k}})}\zeta}\big|_{t=0}}{\sigma}=\dcp{(\Hess F)_\zeta(\ad_{\xi^{\g{k}}}\zeta)}{\sigma}=\ssp{\#(\Hess F)_\zeta(\ad_{\xi^{\g{k}}}\zeta)}{\sigma}\ .
    \end{split}
\end{equation}
Therefore, since $\ad_{(\nabla f)_{\mu_\g{p}(z)}}\xi^{\g{k}}\in\g{p}$, and $F$ is $(\g{k},\g{p})$-compatible at $\zeta=\mu_\g{p}(z)$, we have proved that  

$$\ad_{(\nabla f)_{\mu_\g{p}(z)}}\xi^{\g{k}}=-\#(\Hess F)_\zeta(\ad_{\xi^{\g{k}}}\zeta)\ ,$$

\noindent Using Definition \ref{def1}, we conclude  

$$\nsp{\ad_{(\nabla f)_{\mu_\g{p}(z)}}\xi^\g{k}}{\eta^\g{p}}_{\mu_\g{p}(z)}=\ssp{\ad_{\mu_\g{p}(z)}\xi^{\g{k}}}{\eta^\g{p}}\ .$$

Now we consider the term $\ad_{(\nabla f)_{\mu_\g{p}(z)}}\xi^\g{p}=-\ad_{\xi^\g{p}}(\nabla f)_{\mu_\g{p}(z)}$. First we notice that for all $\sigma\in\g{k}$ we have 
$$\ssp{\ad_{(\nabla f)_\zeta}\xi^\g{p}}{\sigma}={-}\ssp{\ad_{\xi^\g{p}}(\nabla f)_\zeta}{\sigma}=\qquad\qquad\qquad\quad\   $$

$$=-\ssp{(\nabla f)_\zeta}{\ad_{\xi^\g{p}}\sigma}={-}\ssp{(\nabla F)_\zeta}{\ad_{\xi^\g{p}}\sigma}\ ,\qquad\qquad\ \ \ $$

\noindent where the last equality is due to the fact that $\ad_{\xi^\g{p}}\sigma\in \g{p}$ and by hypotheses $f=F|_{\g{p}}$. Let us denote for simplicity $\alpha:=i\xi^\g{p}\in \g{u}$ and $\beta:=i\sigma\in i\g{u}$. Then, using again Lemma \ref{lemma1}, we get

\begin{equation}\label{Eq:CriticalEq}
    \begin{split}
        &-\ssp{(\nabla F)_\zeta}{\ad_{\xi^\g{p}}\sigma}=\ssp{(\nabla F)_\zeta}{\ad_{\alpha}\beta}=\dcp{(\di F)_\zeta}{\ad_{\alpha}\beta}=\\
        &=\dcp{(\di F)_\zeta}{\der{}{t}\Ad_{\exp(t\alpha)}\beta\big|_{t=0}}=\\
        &=\der{}{t}\left(\dcp{(\di F)_{\Ad_{\exp(t\alpha)}\zeta}}{\Ad_{\exp(t\alpha)}\beta}\right)\bigg|_{t=0}-\dcp{\der{}{t}(\di F)_{\Ad_{\exp(t\alpha)}\zeta}\big|_{t=0}}{\beta}=\\
        &=-\dcp{\der{}{t}(\di F)_{\Ad_{\exp(t\alpha)}\zeta}\big|_{t=0}}{\beta}=-\dcp{(\Hess F)_\zeta(\ad_\alpha\zeta)}{\beta}=-\ssp{\#(\Hess F)_\zeta \ad_{\alpha}\zeta}{\beta}\ .
        \end{split}
\end{equation}

\noindent Therefore, we conclude that for all $\sigma\in\g{k}$ we have
$$\ssp{\ad_{(\nabla f)_\zeta}\xi^\g{p}}{\sigma}={-\ssp{\#(\Hess F)_\zeta \ad_{\alpha}\zeta}{\beta}} = \ssp{i\#(\Hess F)_\zeta i\ \ad_{\xi^\g{p}}\zeta}{\sigma}\ .$$
Finally, we can rewrite the second term as

\begin{equation}\label{Eq:CritEq2}
    \begin{split}
        &\nsp{\ad_{(\nabla f)_{\mu_{\g{p}}(z)}}\xi^\g{p}}{\eta^\g{k}}_{\mu_{\g{p}}(z)}=\ssp{\ad_{(\nabla f)_{\mu_{\g{p}}(z)}}\xi^\g{p}}{(\#(\Hess \tilde F))_{-i\zeta}^{-1}\eta^\g{k}}=\\
        &=\ssp{i\#(\Hess F)_\zeta i\ \ad_{\xi^\g{p}}\zeta}{(\#(\Hess \tilde F))_{-i\zeta}^{-1}\eta^\g{k}}=-\ssp{\#(\Hess F)_\zeta i\ \ad_{\xi^\g{p}}\zeta}{i(\#(\Hess \tilde F))_{-i\zeta}^{-1}\eta^\g{k}}=\\
        &=-\ssp{\#(\Hess F)_\zeta i\ \ad_{\xi^\g{p}}\zeta}{(\#(\Hess F))_{\zeta}^{-1}i\eta^\g{k}}=-\ssp{i\ \ad_{\xi^\g{p}}\zeta}{i\eta^\g{k}}=\\
        &=\ssp{i \ad_{\zeta}\xi^\g{p}}{i\eta^\g{k}}=\ssp{\ad_{\zeta}\xi^\g{p}}{\eta^\g{k}}\ .  
    \end{split}
\end{equation}

Summing up, always by the Cartan relations (\ref{Eq:CartanRel}), and using the $\ssp{\cdot}{\cdot}$-orthogonality of the splitting $\g{g}=\g{k}\oplus\g{p}$, we conclude that

$$ \nsp{\ad_{(\nabla f)_{\mu_\g{p}(z)}}\xi}{\eta}_{\mu_\g{p}(z)}=\ssp{\ad_{{\mu_\g{p}(z)}}\xi^\g{k}}{\eta^\g{p}}+\ssp{\ad_{{\mu_\g{p}(z)}}\xi^\g{p}}{\eta^\g{k}}=\ssp{\ad_{{\mu_\g{p}(z)}}\xi}{\eta}\ .$$
\end{proof}

\begin{rem}
    Note that in Equation (\ref{Eq:ParteFacile}) we can write $f$ or $F$ interchangeably, since we evaluate and differentiate $f$ inside $\g{p}$ and, by hypotheses, $f=F|_{\g{p}}$. On the contrary, in Equation (\ref{Eq:CriticalEq}), we have to differentiate $F$ along directions that, in general, lie in $i\g{u}$, hence we need to work with an extension of $f$. We further point out that $\#(\Hess F)_\zeta(\g{p})\subseteq \g{p}$ was used to treat the first term of Equation (\ref{Eq:Spliteq}) while $\#(\Hess F)_\zeta(i\g{k})\subseteq i\g{k}$ is crucial for treating the second term of (\ref{Eq:Spliteq}): specifically, it has been used in the second equality of Equation (\ref{Eq:CritEq2}).
\end{rem}

We are now ready to prove Theorem \ref{mainthm}.

\begin{proof} \textit{(Theorem \ref{mainthm})} \textit{Step 1}.
We first prove that the theorem is valid for the particular case of 
$f(\cdot)={\frac 12}\ssp{\cdot}{\cdot}|_{\g{p}}$. By Proposition \ref{P:Diffpz}, and Remark \ref{R:Diffsspz}, $(\nabla f)_{\mu_\g{p}(z)}=\mu_\g{p}(z)\in \g{p}_z=\g{g}_z\cap \g{p}$. In general, given a Cartan decomposition $\g{g}=\g{k}\oplus \g{p}$, we know that for all $\beta\in \g{p}$ the endomorphism $\ad_\beta:\g{g}\to \g{g}$ is diagonalizable. Since $\g{g}_z$ is a Lie subalgebra of $\g{g}$, $\ad_{\mu_\g{p}(z)}$ is a diagonalizable endomorphism of $\g{g}_z$. Let $\xi\in\g{g}_z$ be an eigenvector of $\ad_{\mu_\g{p}(z)}$ of eigenvalue $\lambda\in\R$. Then
\begin{equation}\label{Eq:MomentSSP0}\ssp{\ad_{\mu_\g{p}(z)}\xi}{\xi}=\lambda\ssp{\xi}{\xi}\ .\end{equation}
Write $\xi=\xi^{\g{k}}+\xi^\g{p}$ {with $\xi^{\g{k}}\in \g{k}, \xi^\g{p}\in \g{p}$} with 
respect to the decomposition $\g{g}=\g{k}\oplus \g{p}$. 
We have already observed that 
\begin{equation*}\ssp{\ad_{{\mu_\g{p}(z)}}\xi}{\xi}=\ssp{\ad_{{\mu_\g{p}(z)}}\xi^\g{k}}{\xi^\g{p}}+\ssp{\ad_{{\mu_\g{p}(z)}}\xi^\g{p}}{\xi^\g{k}}\ .\end{equation*}
We further notice that $$\ssp{\ad_{{\mu_\g{p}(z)}}\xi^\g{p}}{\xi^\g{k}}=-\ssp{\ad_{\xi^\g{p}}\mu_{\g{p}}(z)}{\xi^\g{k}}=\ssp{\mu_{\g{p}}(z)}{\ad_{\xi^\g{k}}\xi^\g{p}}=-\ssp{\ad_{\xi^\g{k}}\mu_{\g{p}}(z)}{\xi^\g{p}}\ .$$
Therefore, Equation (\ref{Eq:MomentSSP0}) is equivalent to $$\lambda\ssp{\xi}{\xi}=-2\ssp{\ad_{\xi^\g{k}}\mu_{\g{p}}(z)}{\xi^\g{p}}\ .$$ Moreover, we have 
\begin{equation*}
    \begin{split}
        &\ssp{\ad_{\xi^\g{k}}\mu_{\g{p}}(z)}{\xi^\g{p}}=\der{}{t}\ssp{\Ad_{\exp(t\xi^{\g{k}})}(\mu_\g{p}(z))}{\xi^\g{p}}\bigg|_{t=0}=\\
        &=\der{}{t}\ssp{\mu_\g{p}(\exp(t\xi^{\g{k}}).z}{\xi^\g{p}}\bigg|_{t=0}=\ssp{(\di\mu_{\g{p}})_z(\xi_Z^\g{k}(z))}{\xi^\g{p}}=\\
        &=Q_z(\xi_Z^\g{k}(z),\xi_Z^\g{p}(z))\ ,
    \end{split}
\end{equation*}
\noindent where, in the last equality, we used Item $(ii)$ of Proposition \ref{P:Prop1}.

Finally, since $\xi^\g{p}_Z(z)+\xi^{\g{k}}_Z(z)=\xi_Z(z)=0$, because $\xi\in\g{g}_z$, we obtain
$$\lambda\ssp{\xi}{\xi}=-2Q_z(\xi^{\g{p}}_Z(z),\xi^{\g{k}}_Z(z))=2Q_z(\xi^{\g{p}}_Z(z),\xi^{\g{p}}_Z(z))\geq 0\ ,$$
from which we conclude that $\lambda\geq 0$.\\

\noindent\textit{Step 2.} Consider now a general $f$ as in the hypothesis of Theorem \ref{mainthm}. By Proposition \ref{P:Diffpz}, $(\nabla f)_{\mu_\g{p}(z)}\in \g{p}_z$, and then $\ad_{(\nabla f)_{\mu_\g{p}(z)}}:\g{g}_z\to\g{g}_z$ is a diagonalizable endomorphism. Let $F$ be an $\Ad_U$-invariant, strictly convex extension of $f$ to $i\g{u}$ which is $(\g{k},\g{p})$-compatible at $\mu_\g{p}(z)\in i\g{u}$, and let $\nsp{\cdot}{\cdot}_{\mu_\g{p}(z)}:\g{u}^\C\times\g{u}^\C\to \R$ be the shifted scalar product of Definition \ref{def1}. From the computations in Step 1 and using Lemma \ref{lemma2}, we conclude that if $\xi\in\g{g}_z$ is an eigenvector of $\ad_{(\nabla f)_{\mu_\g{p}(z)}}$ of eigenvalue $\lambda$, we have 
\begin{equation}\label{autovalore}
\lambda \nsp{\xi}{\xi}{_{\mu_\g{p}(z)}}=\nsp{\ad_{(\nabla f)_{\mu_\g{p}(z)}}\xi}{\xi}_{\mu_\g{p}(z)}=\ssp{\ad_{\mu_\g{p}(z)}\xi}{\xi}=2Q|_z(\xi^{\g{p}}_Z(z),\xi^{\g{p}}_Z(z))\geq 0\ ,\end{equation}
showing that also in this case $\lambda\geq 0$.

{We 
 now prove that $\ker\left(\ad_{(\nabla f)_{\mu_\g{p}(z)}}\right) = \g{k}_z\oplus\g{p}_z$. 
 We note that if $\xi\in\g{g}_z$ is in the kernel of $\ad_{(\nabla f)_{\mu_\g{p}(z)}}$ then, by \eqref{autovalore}, we have $\xi_Z^\g{p}(z)=0$, hence $\xi_Z^\g{k}(z)=0$ and $\xi \in 
    \g{k}_z\oplus \g{p}_z$. 
    
Viceversa, let $Y\in \g{k}_z$ and note that the $K$-equivariance of $\mu_{\g{p}}$ implies $[Y,\mu_{\g{p}}(z)]=0$ and hence $\Ad_{\exp(tY)}\mu_{\g{p}}(z)=\mu_{\g{p}}(z)$ for all $t\in \R$. Moreover, by the $\Ad(K)$-invariance of $f$ we have 
$$\Ad_{\exp(tY)}(\nabla f)_{\mu_{\g{p}}(z)} = (\nabla f)_{\Ad_{\exp(tY)}\mu_{\g{p}}(z)} = (\nabla f)_{\mu_{\g{p}}(z)}\ ,$$
and therefore $[Y,(\nabla f)_{\mu_{\g{p}}(z)}] =0$, showing that $[\g{k}_z,(\nabla f)_{\mu_{\g{p}}(z)}] =0$.

We now take $X\in \g{p}_z$. We clearly have that $[X,(\nabla f)_{\mu_{\g{p}}(z)}] \in \g{k}_z$ by the Cartan relation and the fact that both vectors belong to the isotropy subalgebra. Therefore,
$$\left[\left[X,(\nabla f)_{\mu_{\g{p}}(z)}\right],(\nabla f)_{\mu_{\g{p}}(z)}\right] =0\ .$$
As $\ad_{(\nabla f)_{\mu_{\g{p}}(z)}}$ is diagonalizable, the kernel of its square coincides with its kernel and therefore $X\in \ker\left(\ad_{(\nabla f)_{\mu_{\g{p}}(z)}}\right)$, showing that $[\g{p}_z,(\nabla f)_{\mu_{\g{p}}(z)}] =0$ and proving our claim.
 }\end{proof}

\begin{rem}\label{R:SignOp}
    In the paper \cite{LSW25} the authors use the same convention as ours for the momentum map and for the Riemannian metric. The difference between the signs of the operators in (\ref{Eq1:CMdec}) and (\ref{Eq:CMdec2}) in the case of $G=U^\C$ can be explained as follows. Let $f:\g{u}^*\to \R$ be an $\Ad_U^*$-invariant function. Then,
    if $\alpha\in \g{u}^*$, $(\di f)_\alpha\in T^*_{\alpha}\g{u}^*\simeq \g{u}$. To compare with our setting, we need a function $i\g{u}\to\R$, because we have to compose it with $\mu_{\g{p}}:Z\to \g{p}=i\g{u}$.
    It can be easily checked that $F:i\g{u}\to \R$, defined by $F(\zeta):=f(\#^{-1}(i\zeta))$ for all $\zeta\in i\g{u}$, is in fact convex and $\Ad_U$-invariant. Moreover, for such function, we have the relation \begin{equation}\label{Eq:Segno}\left(\nabla F\right)_{{\mu_\g{p}}}= -i (\di f)_{\mu}\ .\end{equation} Now, by Theorem \ref{mainthm} (or by Corollary \ref{C:CorollLSW25}), $\ad_{(\nabla F)_{\mu_\g{p}(z)}}$ has non negative eigenvalues. By (\ref{Eq:Segno}), this is in accordance with the fact that in (\ref{Eq1:CMdec}) the operator $\ad_{i(\di f)_\mu}$ has non positive eigenvalues. 
\end{rem}

\section{Application}
 We can use Theorem \ref{mainthm} to prove the following Theorem \ref{mainthm2} which was known in the case $G=U^\C$ (see Theorem 13.5 in \cite{GRS24}). It is noteworthy that the proof of this result for a general reductive group $G$ is very different from the specific case $G=U^\C$ (see Remark \ref{R:Salamon}).

 \begin{mainthm}\label{mainthm2}
 Let $(Z,J,\omega)$ be a compact K\"ahler manifold which is acted on isometrically by a compact connected group $U\subset \text{U}(n)$. Let $G$ be a real compatible subgroup of $U^{\mathbb C}$. Let $\mu_\g{p}:Z\to \g{p}$ be the gradient map associated with an $\Ad_U^*$-equivariant momentum map $\mu:Z\to \g{u}^*$, and let $f:\g{p}\to\R$ be an $\Ad_K$-invariant, strictly convex function. If $z\in Z$ is a critical point of $f\circ \mu_\g{p}:Z\to \R$, and $f$ admits a strictly convex, $\Ad_U$-invariant extension to $i\g{u}$, which is $(\g{k},\g{p})$-compatible at $\mu_\g{p}(z)$, then
$$K_z^o= G_z^o\cap K$$
and $K_z^o$ is a maximal compact subgroup of $G_z^o$.

 \end{mainthm}

\begin{proof} 
By Theorem \ref{mainthm} we have the decomposition 
$$\g{g}_z = \g{k}_z \oplus \g{p}_z \oplus\bigoplus_{\lambda >0}\g{g}_\lambda$$
as eigenspaces of the adjoint map $\ad_{{(\nabla f)}_{\mu_{\g{p}}(z)}}$, where 
$$\g{g}_0 = \g{k}_z+\g{p}_z,\quad \g{e} = \bigoplus_{\lambda >0}\g{g}_\lambda\ .$$
We denote by $G_0$ and $E$ the connected subgroups of $G_z^o$ with Lie algebra $\g{g}_0$ and $\g{e}$ respectively.
\par\medskip
\noindent {\bf Claim 1.}\ $E$ is closed and simply connected.\par\medskip
\noindent We consider the linear subgroup {$\tilde E:= {\rm{Ad}}_{G_z^o}(E)\subseteq {\rm{Ad}}_{G_z^o}(G_z^o)$}. As $E$ is nilpotent and normal, every endomorphism $\ad_v$, $v\in \g{e}$, is nilpotent and the group $\bar E$ is unipotent. We know that a linear unipotent subgroup is closed and diffeomorphic to an euclidean space (as the exponential map is a diffeomorphism when restricted to the set of upper triangular matrices with all zeros on the diagonal). Therefore ${\rm{Ad}}_{G_z^o}^{-1}(\tilde E)$ is a closed subgroup given by $E\cdot Z$, where $Z$ is the center of $G_z^o$. Now we note that the Lie algebra $\g{z}$ of $Z$ is contained in $\g{g}_0$: indeed $[\g{z},{(\nabla f)}_{\mu_{\g{p}}(z)}]=0$ implies that $\g{z}\subseteq \g{g}_0$. Therefore $E\cap Z$ is discrete and ${\rm{Ad}}_{G_z^o}:E\to \tilde E$ is a covering. As $\tilde E$ is simply connected, we see that $E$ is simply connected and $Z\cap E=\{1\}$. It then follows that {$(E\cdot Z)^o\cong E\times Z^o$} and therefore $E$ is closed {in $(E\cdot Z)^o$, hence in $G_z^o$ and in $G$.}\par\medskip
\noindent {\bf Claim 2.}\ A maximal compact subgroup of $G_0$ is conjugate to $K_z^o$.\par\medskip
We know that $G=K\cdot\exp(\g{p})$ and moreover $K_z^o\cdot \exp(\g{p}_z)\subseteq G_0$ is a closed submanifold of dimension $\dim \g{g}_0$, hence it is open on $G_0$. By connectedness of $G_0$, we have $G_0=K_z^o\cdot \exp(\g{p}_z)$. Let $K'$ be a maximal compact subgroup of $G_0$ containing $K_z^o$. Now $K'$ is connected and if $K'$ is strictly bigger than $K_z^o$, then there exists a non-zero $v\in \g{k}'\cap \g{p}_z$. Then the closure $W$ of $\exp(\mathbb R\cdot v)$ in $G_z^o$ is compact and contained in the closed submanifold $\exp(\g{p}_z)$. On the other hand, $\exp:\g{p} \to \exp(\g{p})$ is a diffeomorphism and therefore $(\exp|_{\g{p}_z})^{-1}(W)$ is compact, while it contains the full line $\mathbb R\cdot v$, a contradiction. Therefore $\g{k}'=\g{k}_z$ and $K'=K_z^o$.\par\medskip
\noindent {\bf Claim 3.}\ $E\cap G_0=\{1\}$.\par\medskip
Given $g\in E\cap G_0$, we write $g=\exp(v)$ for some $v\in\g{e}$. Let $\zeta:=\mu_{\g{p}}(z)$. If $u=\exp\left({(\nabla f)}_{\zeta}\right)$, then $u$ centralizes $G_0$, hence
$$ \exp({\rm{Ad}}_uv) =\exp\left(\left(e^{\ad_{{(\nabla f)}_{\zeta}}}\right)v\right)= \exp(v)\ .$$
As $\exp|_{\g{e}}$ is injective and $e^{\ad_{{(\nabla f)}_{\zeta}}}$ has eigenvalues strictly bigger than $1$, we see that $v=0$ and our claim follows.

Now let $C$ be a compact connected subgroup of $G_z^o$ containing $K_z^o$. As $E$ is closed and normal, we consider the projection $\pi:G_z^o\to G_z^o/E := \bar G$. Then $\mathrm{Lie(\bar G)}\cong\g{g}_0$ and the restriction $\pi|_{G_0}:G_0\to \bar G$ is an isomorphism as $\pi_*|_{\g{g_0}}$ is an isomorphism and $G_0\cap E=\{1\}$. Now $\pi(C)$ is compact and therefore
$(\pi|_{G_0})^{-1}(\pi(C))$ is compact as well and it contains $K_z^o$ (as $C\supseteq K_z^o$
implies $\pi(C)\supseteq \pi(K_z^o)=(\pi|_{G_0}(K_z^o)$). From Claim 2 we have that $(\pi|_{G_o})^{-1}(\pi(C))= K_z^o$, hence 
$${\rm{Lie}}((\pi|_{G_o})^{-1}(\pi(C))) = (\pi|_{G_o})^{-1}_*({\rm{Lie}}(\pi(C)) = \g{k}_z$$
and therefore,  
$$\pi_*\g{c} = {\rm{Lie}}(\pi(C)) = (\pi|_{G_o})_*\g{k}_z= \pi_*\g{k}_z\ ,$$
as $\g{k}_z\subset\g{g}_o$.
This means that $\g{c} \equiv \g{k}_z\ ({\rm{mod}}\ \g{e})$. Now, if we suppose that $\g{k}_z\subsetneq \g{c}$, then there exists a non-zero $v\in \g{c}\cap \g{e}$ and the closure of
$\{\exp(tv)|\ t\in\mathbb R\}$ is a compact subset contained in the closed $E$. On the other hand $E$ is a simply connected nilpotent group and does not contain non-trivial compact subgroups. Therefore $\g{c}=\g{k}_z$, hence $C=K_z^o$. If now $C'$ is a maximal compact subgroup containing $G_z^o\cap K\supseteq K_z^o$, then $C'$ is conjugate to $K_z^o$ and therefore $\dim \g{c}'=\dim \g{k}_z$. As $C'$ is connected, we have $C'=K_z^o$ and therefore
$$K_z^o=G_z^o\cap K$$
is a maximal compact subgroup.
\end{proof}

\begin{rem}\label{R:Salamon}
    We would like to briefly dwell on the fact that if we restrict ourselves to the case $G=U^\C$, the proof of Theorem \ref{mainthm2} is substantially simplified. The reason for this simplification is due to the following fact, which is used in the proof of Theorem 13.5 in \cite{GRS24}. Let $\g{c}$ be the Lie algebra of a maximal compact subgroup of $G_z^o=(U^\C_z)^o$ containing $U_z^o=K^o_z$. Let $\pi:\g{c}\to \g{u}_z$ be the orthogonal projection and consider the splittings $\g{c}=Z(\g{c})\oplus [\g{c},\g{c}]$, $\g{u}_z=Z(\g{u}_z)\oplus [\g{u}_z,\g{u}_z]$ of the Lie algebras $\g{c},\g{u}_z$ into their respective centers and semisimple parts.
Then, one has $$\pi(Z(\g{c}))= Z(\g{u}_z)\quad \text{and}\quad \pi([\g{c},\g{c}])= [\g{u}_z,\g{u}_z]\ .$$ The proof of this fact can be divided into three steps.
\begin{enumerate}
    \item[(a)] First, we prove that $\pi(Z(\g{c}))\subseteq Z(\g{u}_z)$. Let $\xi\in Z(\g{c})$ and decompose it as $\xi=\pi(\xi)+\xi^\bot$, $\xi^\bot\in\g{u}_z^\bot$. Then given any $\alpha\in\g{u}_z$, for all $\beta\in\g{u}_z$ we have $$\ssp{[\pi(\xi),\alpha]}{\beta}=\ssp{[\xi-\xi^\bot,\alpha]}{\beta}=-\ssp{\xi^\bot}{[\alpha,\beta]}=0\ ,$$ meaning that $[\pi(\xi),\alpha]=0$, and we conclude that $\pi(\xi)\in Z(\g{u}_z)$; \\

    \item[(b)]  We now show that $\pi([\g{c},\g{c}]))\subseteq [\g{u}_z,\g{u}_z]$. Let $\xi_1,\xi_2\in\g{c}$ and let $\xi_j=\zeta_j+i\eta_j+\varepsilon_j$, for $j=1,2$ be the decomposition of $\xi_j$ with respect to the splitting (\ref{Eq1:CMdec}), i.e., $\zeta_j,\eta_j\in \g{u}_z$ and $\varepsilon_j\in\g{e} = \bigoplus_{\lambda<0}\g{u}^\C_{z,\lambda}$. Then $\pi([\xi_1,\xi_2])=[\zeta_1,\zeta_2]-[\eta_1,\eta_2]\in[\g{u}_z,\g{u}_z]$, because $[\g{u}_z,i\g{u}_z]\subseteq i\g{u}_z$, $[\g{e},\g{e}]\subseteq\g{e}$ and $[\g{u}_z,\g{e}]\subseteq \g{e}$;\\
    \item[(c)] By item (a) and (b), and since the projection is surjective, it must be $\pi(Z(\g{c}))= Z(\g{u}_z)$ and $\pi([\g{c},\g{c}])= [\g{u}_z,\g{u}_z]$.
\end{enumerate}

We observe that while item (a) holds also for general $\g{k}_z\subset \g{u}_z$, item (b) fails. Indeed, by repeating the same arguments, using decomposition (\ref{Eq:CMdec2}), we must deal with the fact that, in general, $[\g{p}_z,\g{p}_z]\not\subseteq [\g{k}_z,\g{k}_z]$, but only $[\g{p}_z,\g{p}_z]\subseteq \g{k}_z$.
\end{rem}

\end{document}